\documentclass[12pt,a4paper]{amsart}
\usepackage{fullpage,amssymb}
\usepackage{hyperref}
\DeclareMathOperator{\sgn}{sgn}
\DeclareMathOperator{\ord}{ord}
\DeclareMathOperator{\lcm}{lcm}
\newcommand{\Z}{\mathbb{Z}}

\newtheorem*{lemma}{Lemma}
\begin{document}
\title{Cracking the problem with 33}
\author{Andrew R. Booker}
\thanks{This work was carried out using the computational
facilities of the Advanced Computing Research Centre, University of
Bristol, \url{http://www.bris.ac.uk/acrc/}.
The author was partially supported by EPSRC Grant \texttt{EP/K034383/1}.
No data were created in the course of this study.}

\address{School of Mathematics, University of Bristol,
University Walk, Bristol, BS8 1TW, United Kingdom}
\email{andrew.booker@bristol.ac.uk}
\begin{abstract}
Inspired by the Numberphile video ``The uncracked problem with 33'' by
Tim Browning and Brady Haran \cite{Numberphile}, we investigate solutions to
$x^3+y^3+z^3=k$ for a few small values of $k$.
We find the first known solution for $k=33$.
\end{abstract}
\maketitle
\section{Introduction}
Let $k$ be a positive integer with $k\not\equiv\pm4\pmod{9}$. Then Heath-Brown
\cite{Heath-Brown} has conjectured that there are infinitely many triples
$(x,y,z)\in\Z^3$ such that
\begin{equation}\label{eq:main}
k=x^3+y^3+z^3.
\end{equation}
Various numerical investigations of \eqref{eq:main} have been carried out,
beginning as early as 1954 \cite{MW}; see \cite{BPTY} for a thorough
account of the history of these investigations up to 2000.
The computations performed since that time have been dominated
by an algorithm due to Elkies \cite{Elkies}. The latest that we are aware of
is the paper of Huisman \cite{Huisman} (based on the implementation by
Elsenhans and Jahnel \cite{EJ}),
which determined all solutions to \eqref{eq:main} with $k<1000$
and $\max\{|x|,|y|,|z|\}\le 10^{15}$. In particular, Huisman reports
that solutions are known for all but 13 values of $k<1000$:
\begin{equation}\label{eq:exceptions}
33,\;42,\;114,\;165,\;390,\;579,\;627,\;633,\;732,\;795,\;906,\;921,\;975.
\end{equation}

Elkies' algorithm works by finding rational points near the Fermat
curve $X^3+Y^3=1$ using lattice basis reduction; it is well suited to
finding solutions for many values of $k$ simultaneously. In this paper we
describe a different approach that is more efficient when $k$ is fixed.
It has the advantage of provably finding all solutions with a bound on
the \emph{smallest} coordinate, rather than the largest as in Elkies'
algorithm. This always yields a nontrivial expansion of the search range
since, apart from finitely many exceptions that can be accounted for separately,
one has
$$
\max\{|x|,|y|,|z|\}>\sqrt[3]{2}\min\{|x|,|y|,|z|\}.
$$
Moreover, empirically it is often the case that one of the variables
is much smaller than the other two, so we expect the gain to be even
greater in practice.

Our strategy is similar to some earlier approaches (see especially
\cite{HLT}, \cite{Bremner}, \cite{KTS} and \cite{BPTY}), and is based
on the observation that in any solution, $k-z^3=x^3+y^3$ has $x+y$ as a
factor. Our main contribution over the earlier investigations is to note
that with some time-space tradeoffs, the running time is very nearly
linear in the height bound, and it is quite practical when implemented
on modern 64-bit computers.

In more detail, suppose that $(x,y,z)$ is
a solution to \eqref{eq:main}, and assume without loss of generality
that $|x|\ge|y|\ge|z|$. Then we have
$$
k-z^3=x^3+y^3=(x+y)(x^2-xy+y^2).
$$
If $k-z^3=0$ then $y=-x$, and every value of $x$ yields a
solution. Otherwise,
setting $d=|x+y|=|x|+y\sgn{x}$, we see that $d$ divides $|k-z^3|$, and
\begin{align*}
\frac{|k-z^3|}{d}&=x^2-xy+y^2=x(2x-(x+y))+y^2\\
&=|x|(2|x|-d)+(d-|x|)^2=3x^2-3d|x|+d^2,
\end{align*}
so that
\begin{equation}\label{eq:solution}
\{x,y\}=\left\{\frac12\sgn(k-z^3)\left(
d\pm\sqrt{\frac{4|k-z^3|-d^3}{3d}}\right)\right\}.
\end{equation}
Thus, given a candidate value for $z$, there is an effective procedure
to find all corresponding values of $x$ and $y$, by running through all
divisors of $|k-z^3|$. Already this basic algorithm finds all solutions
with $\min\{|x|,|y|,|z|\}\le B$ in time $O(B^{1+\varepsilon})$, assuming
standard heuristics for the time complexity of integer factorization.
In the next section we explain how to avoid factoring and
achieve the same ends more efficiently.

\subsection*{Acknowledgements}
I thank Roger Heath-Brown for helpful comments and suggestions.

\section{Methodology}
For ease of presentation, we will
assume that $k\equiv\pm3\pmod{9}$; note that this
holds for all $k$ in \eqref{eq:exceptions}.
Since the basic algorithm described above is reasonable for finding
small solutions, we will assume henceforth that $|z|>\sqrt{k}$.
Also, if we specialize \eqref{eq:main} to solutions with $y=z$,
then we get the Thue equation $x^3+2y^3=k$, which is efficiently solvable.
Using the Thue solver in PARI/GP \cite{pari}, we verify that
no such solutions exist for the $k$ in \eqref{eq:exceptions}.
Hence we may further assume that $y\ne z$.

Since $|z|>\sqrt{k}\ge\sqrt[3]{k}$, we have
$$\sgn{z}=-\sgn(k-z^3)=-\sgn(x^3+y^3)=-\sgn{x}.$$
Likewise, since $x^3+z^3=k-y^3$ and $|y|\ge|z|$,
we have $\sgn{y}=-\sgn{x}=\sgn{z}$.
Multiplying both sides of \eqref{eq:main} by
$-\sgn{z}$, we thus obtain
\begin{equation}\label{eq:XY}
|x|^3-|y|^3-|z|^3=-k\sgn{z}.
\end{equation}
Set $\alpha=\sqrt[3]{2}-1$, and recall that $d=|x+y|=|x|-|y|$.
If $d\ge\alpha|z|$ then
\begin{align*}
-k\sgn{z}&=|x|^3-|y|^3-|z|^3
\ge(|y|+\alpha|z|)^3-|y|^3-|z|^3\\
&=3\alpha(\alpha+2)(|y|-|z|)z^2+3\alpha(|y|-|z|)^2|z|\\
&\ge3\alpha(\alpha+2)|y-z|z^2.
\end{align*}
Since $3\alpha(\alpha+2)>1$, this is incompatible with our assumptions
that $y\ne z$ and $|z|>\sqrt{k}$. Thus we must have $0<d<\alpha|z|$.

Next, reducing \eqref{eq:XY} modulo $3$ and recalling our assumption
that $k\equiv\pm3\pmod{9}$, we see that
$$
d=|x|-|y|\equiv|z|\pmod{3}.
$$
Let $\epsilon\in\{\pm1\}$ be so that $k\equiv3\epsilon\pmod{9}$. Then,
since every cube is congruent to $0$ or $\pm1\pmod{9}$, we
must have $x\equiv y\equiv z\equiv\epsilon\pmod{3}$, so that
$\sgn{z}=\epsilon\left(\frac{|z|}{3}\right)=\epsilon\left(\frac{d}{3}\right)$.
In view of \eqref{eq:solution}, we get a solution to \eqref{eq:main}
if and only if $d\mid z^3-k$ and
$3d(4|z^3-k|-d^3)=3d(4\epsilon\left(\frac{d}{3}\right)(z^3-k)-d^3)$ is a square.

In summary, to find all solutions to \eqref{eq:main} with
$|x|\ge|y|\ge|z|>\sqrt{k}$, $y\ne z$ and $|z|\le B$, it suffices to solve the
following system for each $d\in\Z\cap(0,\alpha B)$ coprime to $3$:
\begin{equation}\label{eq:system}
\begin{aligned}
&\frac{d}{\sqrt[3]{2}-1}<|z|\le B,
\quad\sgn{z}=\epsilon\left(\frac{d}{3}\right),
\quad z^3\equiv k\pmod{d},\\
&3d\left(4\epsilon\!\left(\frac{d}{3}\right)(z^3-k)-d^3\right)=\square.
\end{aligned}
\end{equation}

Our approach to solving this is straightforward: we work through the
values of $d$ recursively by their prime factorizations, and apply the
Chinese remainder theorem to reduce the solution of $z^3\equiv k\pmod{d}$
to the case of prime power modulus, to which standard algorithms apply.
Let $r_d(k)=\#\{z\pmod{d}:z^3\equiv k\pmod{d}\}$ denote the number of cube
roots of $k$ modulo $d$.
By standard analytic estimates, since $k$ is not a cube, we have
$$
\sum_{d\le\alpha B}r_d(k)\ll_k B.
$$
Heuristically, computing the solutions of $z^3\equiv k\pmod{p}$ for
all primes $p\le\alpha B$ can be done with $O(B)$ arithmetic operations on
integers in $[0,\alpha B]$; see e.g.\ the algorithm described in \cite[\S2.9,
Exercise 8]{NZM}. Assuming this, one can see that with Montgomery's
batch inversion trick \cite[\S10.3.1]{Montgomery2}, the remaining effort
to determine the roots of $z^3\equiv k\pmod{d}$ for all positive integers
$d\le\alpha B$ can again be carried out with $O(B)$ arithmetic operations.

Thus, we can work out all $z$ satisfying the first line of
\eqref{eq:system}, as a union of arithmetic progressions, in linear time.
To detect solutions to the final line, it is crucial to have
a quick method of determining whether
$\Delta:=3d\left(4\epsilon\!\left(\frac{d}{3}\right)(z^3-k)-d^3\right)$
is a square. We first note that for fixed $d$ this condition reduces to
finding an integral point on an elliptic curve; specifically,
writing $X=12d|z|$ and $Y=(6d)^2|x-y|$, from \eqref{eq:solution}
we see that $(X,Y)$ lies on the Mordell curve
\begin{equation}\label{eq:mordell}
Y^2=X^3-2(6d)^3\left(d^3+4\epsilon\left(\frac{d}{3}\right)k\right).
\end{equation}
Thus, for fixed $d$ there are at most finitely many solutions, and they
can be effectively bounded. For some small values of $d$ it is practical
to find all the integral points on \eqref{eq:mordell} and check whether
any yield solutions to \eqref{eq:main}. For instance, using the integral
point functionality in Magma \cite[\S128.2.8]{magma}, we verified that
there are no solutions for $k$ as in \eqref{eq:exceptions} and
$d\le40$, except possibly for $(k,d)\in\{(579,29),(579,34),(975,22)\}$.

Next we note that some congruence and divisibility constraints come for free:
\begin{lemma}
Let $z$ be a solution to \eqref{eq:system}, let $p$ be a prime
number, and set $s=\ord_p{d}$, $t=\ord_p(z^3-k)$. Then:
\begin{enumerate}
\item[(i)]$z\equiv\frac43k(2-d^2)+9(k+d)\pmod{18}$;
\item[(ii)]if $p\equiv2\pmod{3}$ then $t\le 3s$;
\item[(iii)]if $t\le 3s$ then $s\equiv t\pmod{2}$;
\item[(iv)]if $\ord_p{k}\in\{1,2\}$ then $s\in\{0,\ord_p{k}\}$.
\end{enumerate}
\end{lemma}
\begin{proof}
Let $\Delta=3d\left(4\epsilon\!\left(\frac{d}{3}\right)(z^3-k)-d^3\right)$.
Writing $\delta=\left(\frac{d}{3}\right)$, we have
$|z|\equiv d\equiv\delta\pmod{3}$. Observing that
$(\delta+3n)^3\equiv \delta+9n\pmod{27}$, modulo $27$ we have
\begin{align*}
\frac{\Delta}{3d}&=
4\epsilon\delta(z^3-k)-d^3=4|z|^3-d^3-4\epsilon\delta k\\
&\equiv4[\delta+3(|z|-\delta)]-[\delta+3(d-\delta)]-4\epsilon\delta k
=3(4|z|-d)-\delta[18+4(\epsilon k-3)]\\
&\equiv3(4|z|-d)-d[18+4(\epsilon k-3)]
=12|z|-9d-4\epsilon dk\\
&\equiv 3|z|-4\epsilon dk.
\end{align*}
This vanishes modulo $9$, so in order for $\Delta$ to be a square, it
must vanish mod $27$ as well. Hence
$$
z=\epsilon\delta|z|\equiv\frac{4\delta dk}{3}
\equiv\frac{4(2-d^2)k}{3}\pmod{9}.
$$
Reducing \eqref{eq:main} modulo $2$ we see that $z\equiv k+d\pmod{2}$,
and this yields (i).

Next set $u=p^{-s}d$ and $v=p^{-t}\epsilon\delta(z^3-k)$, so that
$$
\Delta=3\bigl(4p^{s+t}uv-p^{4s}u^4\bigr).
$$
If $3s<t$ then $p^{-4s}\Delta\equiv-3u^4\pmod{4p}$,
but this is impossible when $p\equiv2\pmod{3}$, since $-3$ is not a
square modulo $4p$. Hence we must have $t\le 3s$ in that case.

Next suppose that $t\le 3s$. We consider the following cases, which
cover all possibilities:
\begin{itemize}
\item If $p=3$ then $s=t=0$, so $s\equiv t\pmod{2}$.
\item If $p\ne3$ and $3s>t+2\ord_p{2}$ then
$\ord_p\Delta=s+t+2\ord_p2$, so $s\equiv t\pmod{2}$.
\item If $3s\in\{t,t+2\}$ then $s\equiv t\pmod{2}$.
\item If $p=2$ and $3s=t+1$ then
$2^{-4s}\Delta=3\bigl(2uv-u^4\bigr)\equiv3\pmod{4}$,
which is impossible.
\end{itemize}
Thus, in any case we conclude that $s\equiv t\pmod{2}$.

Finally, suppose that $p\mid k$ and $p^3\nmid k$.
If $s=0$ then there is nothing to prove, so assume otherwise.
Since $d\mid{z^3-k}$, we must have $p\mid z$, whence
$$0<s\le t=\ord_p(z^3-k)=\ord_p{k}<3s.$$
By part (iii) it follows that $s\equiv\ord_p{k}\pmod{2}$, and thus
$s=\ord_p{k}$.
\end{proof}
Thus, once the residue class of $z\pmod{d}$ is fixed, its residue modulo
$\lcm(d,18)$ is determined. Note also that conditions (ii) and (iii)
are efficient to test for $p=2$.

However, even with these optimizations there are $\gg B\log{B}$
pairs $d,z$ satisfying the first line of \eqref{eq:system} and
conclusions (i) and (iv) of the lemma.
To achieve better than $O(B\log{B})$
running time therefore requires eliminating some values of $z$ from the start.
We accomplish this with a standard time-space tradeoff.
To be precise, set $P=3(\log\log{B})(\log\log\log{B})$, and let
$M=\prod_{5\le p\le P}p$
be the product of primes in the interval $[5,P]$.
By the prime number theorem, we have
$\log{M}=(1+o(1))P$.
If $\Delta$ is a square, then
for any prime $p\mid M$ we have
\begin{equation}\label{eq:legendre}
\left(\frac{\Delta}{p}\right)
=\left(\frac{3d}{p}\right)
\left(\frac{|z|^3-c}{p}\right)\in\{0,1\},
\end{equation}
where $c\equiv\epsilon\left(\frac{d}{3}\right)k+\frac{d^3}{4}\pmod{M}$.
When $\lcm(d,18)\le\alpha B/M$, we first compute this
function for every residue class $|z|\pmod{M}$, and select only those
residues for which \eqref{eq:legendre} holds
for every $p\mid M$.
By Hasse's bound, the number of permissible residues is at
most
$$
\frac{M}{2^{\omega(M/(M,d))}}
\prod_{p\mid\frac{M}{(M,d)}}\left(1+O\!\left(\frac1{\sqrt{p}}\right)\right)
=\frac{M}{2^{\omega(M/(M,d))}}e^{O(\sqrt{P}/\log{P})},
$$
and thus the total number of $z$ values to consider is at most
\begin{align*}
\sum_{\lcm(d,18)\le\frac{\alpha B}{M}}&r_d(k)\left[M+
\frac{e^{O(\sqrt{P}/\log{P})}}{2^{\omega(M/(M,d))}}
\frac{\alpha B}{d}\right]
+\sum_{\substack{d\le\alpha B\\\lcm(d,18)>\frac{\alpha B}{M}}}
\frac{r_d(k)\alpha B}{d}\\
&\ll_k B\log{M}+
\frac{e^{O(\sqrt{P}/\log{P})}}{2^{\omega(M)}}
\sum_{g\mid M}\frac{2^{\omega(g)}r_g(k)}{g}
\sum_{d'\le\frac{\alpha B}{9gM}}\frac{r_{d'}(k)\alpha B}{d'}\\
&\ll_k B\log{M}+B\log{B}
\frac{e^{O(\sqrt{P}/\log{P})}}{2^{\omega(M)}}
\prod_{p\mid M}\left(1+\frac{2r_p(k)}{p}\right)\\
&\ll BP+\frac{B\log{B}}{2^{(1+o(1))P/\log{P}}}
\ll B(\log\log{B})(\log\log\log{B}).
\end{align*}

For the $z$ that are not eliminated in this way, we follow a similar
strategy with a few other auxiliary moduli $M'$ composed of larger
primes, in order to accelerate the square testing.  We precompute tables
of cubes modulo $M'$ and Legendre symbols modulo $p\mid M'$, so that
testing \eqref{eq:legendre} is reduced to table lookups.
Only when all of these tests pass do we compute $\Delta$ in
multi-precision arithmetic \cite{GMP} and apply a general square test, and this
happens for a vanishingly small proportion of candidate values.
In fact we expect the number of Legendre tests to be bounded on average,
so in total, finding all solutions
with $|z|\le B$ should require no more than
$O_k\bigl(B(\log\log{B})(\log\log\log{B})\bigr)$
table lookups and arithmetic operations on integers in $[0,B]$.

Thus, when $B$ fits within the machine word size, we expect the running
time to be nearly linear, and this is what we observe in practice for
$B<2^{64}$.

\section{Implementation}
We implemented the above algorithm in \texttt{C}, with a few inline assembly
routines for Montgomery arithmetic \cite{Montgomery1} written by Ben
Buhrow \cite{yafu}, and Kim Walisch's \texttt{primesieve} library
\cite{primesieve} for enumerating prime numbers.

The algorithm is naturally split between values of $d$ with a prime
factor exceeding $\sqrt{\alpha B}$ and those that are
$\sqrt{\alpha B}$-smooth. The former set of $d$ consumes more than
two-thirds of the running time, but is more easily parallelized.
We ran this part on the massively parallel cluster Bluecrystal Phase 3
at the Advanced Computing Research Centre, University of Bristol.
For the smooth $d$ we used a separate small cluster of 32- and 64-core
nodes.

We searched for solutions to \eqref{eq:main} for $k\in\{33,42\}$
and $\min\{|x|,|y|,|z|\}\le10^{16}$, and found the following:
$$
33=8\,866\,128\,975\,287\,528^3+(-8\,778\,405\,442\,862\,239)^3
+(-2\,736\,111\,468\,807\,040)^3.
$$
We also searched for solutions for $k=3$, addressing a question of
Mordell \cite[\S6]{Mordell}. In this case, Cassels \cite{Cassels} observed
that cubic reciprocity forces the additional constraint $x\equiv y\equiv
z\pmod{9}$, and it follows that part (i) of the lemma can be upgraded to a
congruence modulo $162$:
$$
z\equiv 4\left(\frac{d}{3}\right)d+3\bigl(d^2-1\bigr)\pmod{162}.
$$
Despite this added efficiency, we found no solutions,
beyond the known single-digit solutions,
with $\min\{|x|,|y|,|z|\}\le 10^{16}$.

The total computation used approximately 23 core-years over
one month of real time.
\bibliographystyle{amsalpha}
\bibliography{cubes}
\end{document}